\documentclass{amsart}
\usepackage{amsaddr}
\usepackage{fullpage}
\usepackage[utf8]{inputenc}
\usepackage[english]{babel}
\usepackage{amsmath,amsthm,amssymb}
\usepackage{ytableau}
\usepackage{booktabs,ulem}
\usepackage{wrapfig}
\usepackage{tikz}
\usetikzlibrary{tikzmark,decorations.pathreplacing}
\usepackage{multirow}

\usepackage[backend=biber,style=numeric]{biblatex}
\definecolor{magenta}{rgb}{1,0,1}

\newcommand{\countingFuncGenus}{PG}
\newcommand{\setGenus}{\mathcal{PG}}
\newcommand{\countingFuncMult}{PM}
\newcommand{\setMult}{\mathcal{PM}}
\newcommand{\countingFuncFrob}{PF}
\newcommand{\setFrob}{\mathcal{PF}}
\newcommand{\noOnesFunc}{\bar{a}}
\newcommand{\noOnesSet}{\bar{\mathcal{A}}}

\newcommand{\setB}{\mathcal{B}}

\newcommand{\N}{\mathbb{N}}

\theoremstyle{definition}
\newtheorem{theorem}{Theorem}
\newtheorem{prop}[theorem]{Proposition}
\newtheorem{cor}[theorem]{Corollary}

\newtheorem{lemma}[theorem]{Lemma}

\newtheorem{remark}{Remark}

\usepackage{mathtools}

\title{On integer partitions corresponding to numerical semigroups}

\author{Hannah E. Burson}
\address{School of Mathematics\\ University of Minnesota, Twin Cities\\ Minneapolis, MN 55455} 
\email{hburson@umn.edu}

\author{Hayan Nam}
\address{Department of Mathematics\\ Duksung Women's University\\  Seoul, Republic of Korea} 
\email{hnam@duksung.ac.kr}

\author{Simone Sisneros-Thiry}
\address{Department of Mathematics\\ California State University, East Bay \\ Hayward, CA 94542} 
\email{simone.sisnerosthiry@csueastbay.edu}

\date{}

\keywords{partitions, numerical semigroups, Frobenius number, genus, multiplicity}
\subjclass{05A17, 11P81, 20M14}

\begin{document}
\normalem
\maketitle

\begin{abstract}
Numerical semigroups are cofinite additive submonoids of the natural numbers. Keith and Nath illustrated an injection from numerical semigroups to integer partitions  \cite{keith2011partitions}.  We explore this connection between partitions and numerical semigroups with a focus on classifying the partitions that appear in the image of the injection from numerical semigroups. In particular, we count the number of partitions that correspond to numerical semigroups in terms of genus, Frobenius number, and multiplicity, with some restrictions.
\end{abstract}

\section{Introduction}
Numerical semigroups arise as a way to study linear Diophantine equations \cite{RosalesGarciaSanchez}. Specifically, given nonnegative integers $a_1,a_2,\ldots, a_k$, for what integers $d$ are there non-negative integers $x_1,x_2,\ldots,x_k$ such that $a_1x_1+a_2x_2+\ldots+a_kx_k=d$? The set of such integers $d$ is called the \emph{numerical semigroup generated by} $\{a_1,a_2,\ldots, a_k\}$, denoted by $\langle a_1,a_2,\ldots, a_k\rangle$. The largest $d$ with no non-negative integer solution is called the Frobenius number of $\langle a_1,a_2,\ldots, a_k\rangle$ and the number of $d$'s with no non-negative integer solution is called the genus of $\langle a_1,a_2,\ldots, a_k\rangle$. 

Frobenius number and genus have been widely studied. Sylvester \cite{sylvester1884mathematical, sylvester1882subvariants} showed that, for coprime positive integers $a$ and $b$, the Frobenius number of the numerical semigroup $\langle a,b\rangle$ is $ab-a-b$ and the corresponding genus is $\frac{(a-1)(b-1)}{2}$. It is still an open question to find a formula for the Frobenius number and genus of a numerical semigroup generated by three positive integers, in general \cite{curtis1990formulas}. See, for example, \cite{branco2023frobenius,  hellus2021frobenius, nari2012genus, rosales2011frobenius, } to learn more about the Frobenius number and genus of certain numerical semigroups. Recently, many researchers have worked towards understanding the number of numerical semigroups with fixed Frobenius number or fixed genus. For example, see \cite{bras2009bounds, robles2022enumeration, rosales2022enumeration, singhal2022distribution}.

In this paper, we study the connection between numerical semigroups and integer partitions. One motivation for studying integer partitions is their connection to representation theory. For example, the irreducible polynomial representations of $\operatorname{GL}_n(\mathbb{C})$ are indexed by partitions of length at most $n$. Furthermore, partitions of $n$ index the conjugacy classes, and, thus, the number of non-equivalent irreducible complex representations, of $S_n$. For more details on these connections see \cite{FultonHarris} and \cite{Sagan}. These connections are particularly evident through the hook length formula, which states that the dimension of the irreducible representation of $S_n$ corresponding to $\lambda$ is given by $$\frac{n!}{\prod_{(i,j)\in \lambda} h_{(i,j)}(\lambda)},$$ where $h_{(i,j)}(\lambda)$ is the length of the hook corresponding to the cell $(i,j)$ in the Young diagram of $\lambda$ \cite{FrameRobinsonThrall}.  This formula motivates the study of \emph{hook lengths} of integer partitions and partitions avoiding hooks of a given length, which are called core partitions.

In 2011, to study partitions with a given set of hook lengths, Keith and Nath \cite{keith2011partitions} introduced a map from numerical semigroups to integer partitions. Later, Constantin, Houston-Edwards, and Kaplan used this map to answer further questions about integer partitions \cite{constantin2015numerical}. Additionally, they explored the function $S(N)$, which counts the number of partitions with largest hook length $N$ in the image of Keith and Nath's map. They showed that $$\lim_{N\to \infty} \frac{S(N)}{T(N)}=0,$$ where $T(N)$ counts all partitions with largest hook length $N$. Furthermore, they conjectured that if $S'(n)$ counts the number of partitions of size at most $n$ that are in the image of Keith and Nath's map, then $$\lim_{n\to \infty}\frac{S'(n)}{P(n)}=0,$$ where $P(n)$ counts all partitions of size at most $n$.

The main goal of this paper is to continue the study of $S'(n)$. We refine the function $S'(n)$ in three different ways by studying the number of such partitions with a given number of parts, largest hook length, or number of parts of size $1$. These three statistics are related to the numerical semigroup invariants genus,  Frobenius number, and multiplicity, respectively. For the numerical semigroup $S$, the \emph{genus} is the number of nonnegative integers in the complement of $S$; the \emph{Frobenius number} is the largest integer in the complement of $S$;  and the \emph{multiplicity} is the smallest positive integer that is in $S$.

The remainder of the paper is structured as follows. In Section \ref{sec:Preliminaries}, we collect definitions and preliminary results, and we explain Keith and Nath's map from numerical semigroups to partitions. In Sections \ref{sec:genus}, \ref{sec:Frobenius}, and \ref{sec:multiplicity}, we study the number of partitions of size $n$ corresponding to numerical semigroups of a given genus, Frobenius number, and multiplicity, respectively.

\section{Preliminaries}\label{sec:Preliminaries}
\subsection{Definitions}
Let $\mathbb{N}_0$ be the set of nonnegative integers. A subset $S$ of $\mathbb{N}_0$ is a {\it numerical set} if $S$ includes $0$, and $S$ has finite complement in $\mathbb{N}_0$. A numerical set $S$ is a {\it numerical semigroup} if, in addition, $S$ is closed under addition. Consider two examples of numerical sets $S_1 = \{0, 3, 5, 6, 8, \rightarrow \}$ and $S_2 = \{0, 4, 7, 9, \rightarrow \}$, where the notation $x,\rightarrow$ indicates that all positive integers greater than $x$ are included in the numerical set. In our example, $S_1$ is a numerical semigroup, but $S_2$ is not. For a complete introduction to numerical semigroups, see \cite{RosalesGarciaSanchez}.

Define the set $G=\mathbb{N}_0\setminus S$ as the set of gaps of the numerical set S. Then the genus of $S$, written as $g(S)$, is equal to $|G|$ (for example, $g(S_1)=4 $ and $g(S_2) =6 $) and the Frobenius number of $S$, written as $f(S)$, is the maximum element of $G$ ($f(S_1) =7 $ and $f(S_2) = 8$). Furthermore, the multiplicity of a numerical set $S$, $m(S)$ is the smallest nonzero element of $S$ ($m(S_1) = 3$ and $m(S_2) =4 $). Note, in the case when $S$ is clear, we often abbreviate the notation to $g$, $f$, and $m$, respectively.

A {\it partition} $\lambda$ of a positive integer $n$ is a sequence of non-increasing positive integers of finite length $(\lambda_1, \lambda_2, \ldots, \lambda_\ell)$ such that $\sum_{i=1}^\ell \lambda_i = n$. Each $\lambda_i$ is called a {\it part} of $\lambda$. In the case of repeated parts, we use frequency notation $(\lambda_1^{j_1}, \lambda_2^{j_2},\ldots, \lambda_k^{j_k})$, with $\lambda_1 > \lambda_2 > \cdots > \lambda_k$. For example, $(3, 2, 2, 1,1,1) = (3, 2^2, 1^3)$ is a partition of 10 into 6 parts. We define $\ell(\lambda)$ as the length, or number of parts of of the partition $\lambda$ (in our example, $\ell(\lambda) = 6)$, and $|\lambda|$ is the size of the partition (in our example, $|\lambda|= 10$). 

A partition can be represented with a \emph{Ferrers diagram}, which is a left-justified array of cells where the $i$th row from the top contains $\lambda_i$ cells. We will abuse notation and use $\lambda$ to refer to the partition, its multiset of parts, or its Ferrers diagram, depending on context. For each box in the diagram corresponding to an integer partition, the \emph{hook} is the collection of boxes which includes: the box itself, all boxes in the same row to the right of the box (the arm), and all boxes in the same column below the box (the leg). The \emph{hook length} is the number of boxes in the hook. We define $h_{(i,j)}(\lambda)$ to be the hook length of the cell $(i, j)$ in the Ferrers diagram of $\lambda$. 

\subsection{Relationship between Partitions and Numerical Sets}\label{sec:relationship}
In \cite{keith2011partitions}, Keith and Nath showed that every numerical set uniquely defines an integer partition. Consider taking a step ``east'' for each element of the numerical set and a step ``north'' for each integer not in the numerical set. The diagram constructed through this process corresponds to what is referred to as the profile of the partition. This process illustrates a bijection between numerical sets and partitions. 

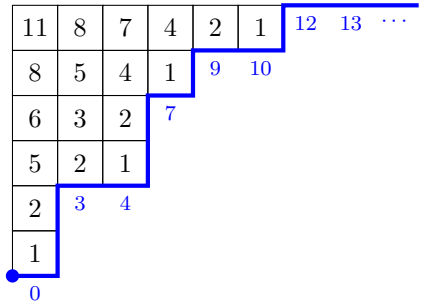
\begin{wrapfigure}{r}{0.36\linewidth}
 \begin{tikzpicture}[scale=0.6]
\begin{scope}[xshift = 16cm]
				\foreach \x in {0,1,2,3,4,5}
					{ \draw (\x,5) rectangle ++(1,1);}
					\foreach \x in {0,1,2,3}
					{ \draw (\x,4) rectangle ++(1,1);}
				\foreach \x in {0,1,2}
					{ \draw (\x,3) rectangle ++(1,1);}
				\foreach \x in {0,1,2}
					{ \draw (\x,2) rectangle ++(1,1);}
					\foreach \y in {0,1}
					{ \draw (0,\y) rectangle ++(1,1);}
				
				\begin{scope}
					\node at (0.5,5.5) {11};
					\node at (1.5,5.5) {8};				
					\node at (2.5,5.5) {7};
					\node at (3.5,5.5) {4};
					\node at (4.5,5.5) {2};
					\node at (5.5,5.5) {1};																
				
					\node at (0.5,4.5) {8};	
					\node at (1.5,4.5) {5};
					\node at (2.5,4.5) {4};
					\node at (3.5,4.5) {1};

					\node at (0.5,3.5) {6};
					\node at (1.5,3.5) {3};
					\node at (2.5,3.5) {2};
				
					\node at (0.5,2.5) {5};
					\node at (1.5,2.5) {2};
					\node at (2.5,2.5) {1};
					
					\node at (0.5,1.5) {2};
					
					\node at (0.5,0.5) {1};
					
				\end{scope}
				
				\begin{scope}[blue, ultra thick]
					\filldraw (0,0) circle [radius = 0.1];
					\draw[font = \footnotesize]
						(0,0) --
						node [below] {0} ++(1,0) --
						node [right] {} ++(0,1) -- 
						node [right] {} ++(0,1) --
						node [below] {3} ++(1,0) --
						node [below] {4} ++(1,0) --
						node [right] {} ++(0,1) --
						node [right] {} ++(0,1) --
						node [below] {7} ++(1,0)--
						node [right] {} ++(0,1) --
						node [below] {9} ++(1,0) --
						node [below] {10} ++(1,0) --
						node [right] {} ++(0,1)--
						node [below] {12} ++(1,0)--
						node [below] {13} ++(1,0)--
						node [below] {$ \cdots$} ++(1,0)  ;
				\end{scope}
			\end{scope}
\end{tikzpicture}
\caption{The partition $(6,4,3^2, 1^2)$ with the profile outlined and the hook lengths and corresponding numerical set labelled.}\label{partitionNumSetCorr}
 \end{wrapfigure}
 
 We can apply this bijection in either direction. As described above, we may begin with a numerical set and determine the partition that corresponds to it. In the other direction, given a partition $\lambda$, we let $S_\lambda$ be the corresponding numerical set.  Define $g_\lambda:=g(S_\lambda), f_\lambda:=f(S_\lambda)$, and $m_\lambda=m(S_\lambda)$. We note the relationship between these invariants of numerical semigroups with properties of the corresponding partition under Keith and Nath's bijection. The genus, $g_\lambda$, is equal to the number of parts of $\lambda$ (note: $g_\lambda = \ell(\lambda)$), and the multiplicity, $m_\lambda$, is one more than the number of parts of size $1$. The Frobenius number of $S_\lambda$, $f_\lambda$ is the largest hook length. Furthermore, the gap set, $G(S_\lambda)$ consists exactly of the hook lengths for the cells in the leftmost column of $\lambda$. As an example, consider the partition $\lambda= (6, 4, 3, 3, 1,1)$, pictured in Figure 1. We have $m_\lambda = 3 = 2+1$, $g_\lambda = 6$ and $f_\lambda = 11$.  We write ``partition with multiplicity $m$," interchangeably with ``partition that corresponds to a numerical set with multiplicity $m$." We will use ``partition with genus $g$" and ``partition with Frobenius number $f$" similarly.

If we restrict the domain to numerical semigroups, we have an injection from numerical semigroups to partitions. Let $\mathcal{P}(n)$ denote the set of partitions of $n$. We discuss the following subsets of the image of the injection from numerical semigroups to partitions, corresponding to the invariants listed above:

\begin{itemize}
    \item $\setGenus(n,g)$, the set of partitions of $n$ that correspond to numerical semigroups with genus $g$ (partitions of $n$ that correspond to numerical semigroups and have exactly $g$ parts)
    \item $\setFrob(n,f)$, the set of partitions of $n$ that correspond to numerical semigroups with Frobenius number $f$ (partitions of $n$ that correspond to numerical semigroups and have largest hook length $f$)
    \item $\setMult(n,m)$ the set of partitions of $n$ that correspond to numerical semigroups with multiplicity $m$ (partitions of $n$ that correspond to numerical semigroups and have exactly $m-1$ parts of size $1$)
\end{itemize}

Further, we define functions that enumerate each of the sets above: $p(n) = |\mathcal{P}(n)|$, $\countingFuncGenus(n,g)=|\setGenus(n,g)|$,  $\countingFuncFrob(n,f)=|\setFrob(n,f)|$, and $\countingFuncMult(n,m)=|\setMult(n,m)|$. 

\subsection{Useful facts about partitions and numerical semigroups.}
We state observations that follow directly from the definitions or the relationship described in Section \ref{sec:relationship}. 
\begin{itemize}
    \item The partitions that correspond to numerical semigroups have at least one part of size 1. This is because the only numerical semigroup that contains $1$ is the nonnnegative integers.
    \item If $F_\lambda < 2m_\lambda$, then $S_\lambda$ is a numerical semigroup. This is because, for $x, y \in S_\lambda$, $2m_\lambda \leq x+y$, so this condition implies that $S_\lambda$ is closed under addition.
    \item The gap between any two parts of a partition $\lambda$ that corresponds with a numerical semigroup is no larger than $m_\lambda$. If there were $m_\lambda$ consecutive integers in $S_\lambda$, every subsequent number must be in $S_\lambda$, which would force the string of $m_\lambda$ consecutive numbers to be larger than $F_\lambda$. 
    
\end{itemize}

\section{Partitions of $n$ that correspond to semigroups of Genus $g$}\label{sec:genus}

Recall that $\setGenus(n,g)$ is the set of partitions of $n$ with $g$ parts corresponding to a numerical semigroup with genus $g$. In this section, we collect identities related to $\countingFuncGenus(n,g)$, which is the number of partitions in $\setGenus(n,g)$. We first show that all partitions of $n$ with $\ge \frac{2}{3}n$ parts correspond to numerical semigroups. Then, we explore some cases where the partition has fewer than $\frac{2}{3}n$ parts, but the difference between the size of the partition and the number of parts is fixed.   

\begin{theorem}\label{GenusPartitionNumbers}
For any positive integer $n$ and $g\geq \frac{2}{3}n$, 
\[ PG(n,g) = p(n-g).\]
\end{theorem}
\begin{proof} 
For a partition $\lambda=\{\lambda_1,\lambda_2,\ldots,\lambda_k,1^{g-k}\} \in \setGenus(n,g)$, where $\lambda_1\ge \lambda_2\ge \ldots\ge \lambda_k\ge 2$, define $$\phi(\lambda):=\{\lambda_1-1,\lambda_2-1,\ldots,\lambda_k-1\}.$$ We will show that $$\phi:\setGenus(n,g)\rightarrow\mathcal{P}(n-g)$$ is a bijection.

Observe that $\phi$ is injective and $\phi(\setGenus(n,g))\subseteq\mathcal{P}(n-g)$. To show that $\phi$ is bijective, we need to show that $\phi^{-1}(\mu)$ corresponds to a numerical semigroup with genus $g$, for any $\mu \in \mathcal{P}(n-g)$. To that end, let $\mu\in \mathcal{P}(n-g)$ and let $\lambda=\phi^{-1}(\mu)$. Note that $\lambda$ is a partition of $n$ with $g$ different parts and $k$ parts of size $>1$, where $k=\ell(\mu)$.

Observe that, for any such partition $\lambda$ with $g$ parts and $k$ parts of size greater than $1$, $m_\lambda=g-k+1$ and $k\le n-g\le \frac{3}{2}g-g=\frac{1}{2}g$. Among all the possible Frobenius numbers of $\lambda$, we get the maximum when $\lambda=\{\lambda_1, 2^{k-1},1^{g-k}\}$. 
Thus, $$f_\lambda\le n-k+1.$$ Since $n\le \frac32 g$ and $k\le \frac12 g$, we can further say
$$
    f_\lambda\le \frac{3}{2}g-k+1\\
    \le\frac32 g-k+1+\left(\frac12 g-k\right)\\
    <2(g-k+1)=2m_\lambda.
$$

Since all numerical sets which have the property that the Frobenius number is less than twice its multiplicity are numerical semigroups, $\phi^{-1}(\mu)$ corresponds to a numerical semigroup and $\phi$ is a bijection.
\end{proof}

We can also find explicit values for $\countingFuncGenus(n,g)$ when $g\approx\frac{2}{3}n$, but $g<\frac{2}{3}n$. In this next theorem, we use the same map as the inverse map above, but we must remove a few exceptional cases where the partitions do not correspond to numerical semigroups. 
\begin{theorem}\label{thm:PNG_large}
 For $n\ge 3$, $$PG(3n-1,2n-1)=p(n)-1$$ and $$PG(3n-2,2n-2)=PG(3n-3,2n-3)=p(n)-2.$$  
\end{theorem}

\begin{proof}
For $n\ge 3$ and each $i=1,2,3$, define the maps
\[
\phi_1:\mathcal{P}(n)\mapsto\setGenus(3n-1,2n-1)\cup \{\{2^n,1^{n-1}\}\}
\]
\[
\phi_2:\mathcal{P}(n)\mapsto\setGenus(3n-2,2n-2)\cup \{\{2^n,1^{n-2}\},\{3,2^{n-2},1^{n-1}\}\}
\]
\[
\phi_3:\mathcal{P}(n)\mapsto\setGenus(3n-3,2n-3)\cup \{\{2^n, 1^{n-3}\}, \{4,2^{n-3},1^{n-1}\}\}
\]
by
\[
\{\lambda_1,\lambda_2,\ldots,\lambda_\ell\}\mapsto \{\lambda_1+1,\lambda_2+1,\ldots, \lambda_\ell+1,1^{2n-i-\ell}\}.
\]

Note that, for each $i=1,2,3$ and a partition $\lambda$ in the domain of $\phi_i$, the size of $\mu_i:=\phi_i(\lambda)$ is $|\lambda|+\ell+(2n-i-\ell)=3n-i$, the length of $\mu_i$ is $\ell+(2n-i-\ell)=2n-i$, and each of the maps $\phi_i$ and $\phi_i^{-1}$ are injective on their domains. To complete the proof, we need to show that $\mu_i$ corresponds to a numerical semigroup with genus $2n-i$. 

In order to show that all of the partitions in the image of the three maps correspond to numerical semigroups, we first consider a partition $\lambda$ of size $n$ such that $\ell:=\ell(\lambda)<n-2$. For each $i=1,2,3$, let $\mu_i=\phi_i(\lambda)$. Then, for each possible $i$, we have $m_{\mu_i}= 2n-i+1-\ell$ and 
\[f_{\mu_i}\le n+1+(2n-i-\ell)=3n+1-i-\ell<4n-2\ell-4\le 2m_{\mu_i}\] 
since $n-\ell>2$. When $n\ge 4$, there are four partitions of $n$ with at least $n-2$ parts, which are $\{1^n\}$,$\{2,1^{n-2}\}$, $\{3,1^{n-3}\}$ and $\{2,2,1^{n-4}\}$. In Table \ref{tab:genusPartitionsManyParts}, we examine which numerical sets correspond to the images of those partitions. 
We observe that $\{0,n,2n+1,\rightarrow\}$, $\{0,n-1,2n,\rightarrow\}$, $\{0,n,2n-1,2n+1,\rightarrow\}$, $\{0,n-2,2n-1,\rightarrow\}$, and $\{0,n,2n-2,2n-1,n+1,\rightarrow\}$ are not numerical semigroups since twice the multiplicity of each set is not in the set. 
When $n=3$, there are only three partitions of $n$ with at least $n-2$ parts and we can similarly identify the exceptions. 
\begin{table}[h]
    \centering
$$
\begin{array}{ccc}
\toprule
    \lambda & \phi_1(\lambda) & S_{\phi_1(\lambda)}  \\ \midrule
    \{1^n\} & \{2^n,1^{n-1}\} &\{0,n,2n+1\rightarrow\}^* \\
    \{2,1^{n-2}\} & \{3,2^{n-2},1^{n}\} & \{0,n+1,2n,2n+2\rightarrow\} \\
    \{3,1^{n-3}\} &\{4,2^{n-3},1^{n+1}\} &\{0,n+2,2n,2n+1,2n+3\rightarrow\} \\
    \{2,2,1^{n-4}\} & \{3,3,2^{n-4},1^{n+1}\} & \{0,n+2,2n-1,2n+2\rightarrow\} 
    \\ \bottomrule
\toprule
    \lambda & \phi_2(\lambda) & S_{\phi_2(\lambda)}  \\ \midrule
    \{1^n\} & \{2^n,1^{n-2}\} &\{0,n-1,2n\rightarrow\}^* \\
    \{2,1^{n-2}\} & \{3,2^{n-2},1^{n-1}\} & \{0,n,2n-1,2n+1\rightarrow\}^* \\
    \{3,1^{n-3}\} &\{4,2^{n-3},1^n\} &\{0,n+1,2n-1,2n,2n+2\rightarrow\} \\
    \{2,2,1^{n-4}\} & \{3,3,2^{n-4},1^n\} & \{0,n+1,2n-2,2n+1\rightarrow\} 
    \\ \bottomrule \toprule
    \lambda& \phi_3(\lambda)& S_{\phi_3(\lambda)} \\ \midrule
    \{1^n\}& \{2^n,1^{n-3}\} & \{0,n-2,2n-1\rightarrow\}^*\\
    \{2,1^{n-2}\} & \{3,2^{n-2},1^{n-2}\} & \{0,n-1,2n-2,2n\rightarrow\}\\
    \{3, 1^{n-3}\} & \{4,2^{n-3},1^{n-1}\} & \{0,n,2n-2,2n-1,2n+1\rightarrow\}^* \\
    \{2,2,1^{n-4}\} & \{3,3,2^{n-4},1^{n-1}\} & \{0,n,2n-3,2n\rightarrow\}\\ \bottomrule
\end{array}
$$
    \caption{Partitions of $n$ into at least $n-2$ parts, with their images under the maps $\phi_1,$ $\phi_2,$ and $\phi_3$. Exceptional cases, where the resulting partition does not correspond to a numerical semigroup are marked with a $^*$ in the third column.}
    \label{tab:genusPartitionsManyParts}
\end{table}
\end{proof}

We continue by connecting two special cases of partitions counted by $\countingFuncGenus(n,g)$ with a fixed value of $n-g$ to other partition-theoretic functions. 
\begin{theorem}\label{GenusBoxedPartitions}
 Let $b(n)$ denote the number of partitions of $n+2$ into at most $n$ parts, each no larger than $n$. Then, for $n\ge 4$ and $k=0,1,2$,
 $$\countingFuncGenus(3n+2-k,2n-k)=b(n)-(k+2).$$
\end{theorem}
\begin{proof}
Let $\mathcal{B}(n)$ be the set of partitions of $n+2$ into at most $n$ parts, each no larger than $n$. Additionally, let $k\in \{0,1,2\}$. For a partition $\lambda= \{\lambda_1, \lambda_2, \ldots, \lambda_{2n-k}\}\in \setGenus(3n+2-k,2n-k)$, define $$\phi:\setGenus(3n+2-k,2n-k)\setminus\{\{n+2,2,1^{2n-k-2}\},\{n+3,1^{2n-k-1}\}\} \to \mathcal{B}(n) $$ by $$\phi(\lambda):=\{\lambda_1-1,\lambda_2-1,\ldots,\lambda_{2n-k}-1\},$$ omitting any resulting parts of size $0$. We will show that $\phi$ is an injection and that $$\left|\setB(n)\setminus\phi(\setGenus(3n+2-k,2n-k)\setminus\{\{n+2,2,1^{2n-k-2}\},\{n+3,1^{2n-k-1}\}\})\right|=k+2.$$ 
Let $\lambda\in \setGenus(3n+2-k,2n-k)\setminus\{\{n+2,2,1^{2n-k-2}\},\{n+3,1^{2n-k-1}\}\}$ and let $\mu=\phi(\lambda)$. Since $|\lambda|=3n+2-k$ and $\ell(\lambda)=2n-k$, $\lambda_1\le n+1$, so $\mu_1\le n$. To show that $\mu\in \setB(n)$, it remains to show that $\ell(\mu)\le n$.  For the sake of contradiction, assume that $\mu$ has more than $n$ parts. Then, $\lambda$ has more than $n$ parts of size at least $2$. Let $S_{\lambda}$ be the numerical set corresponding to the partition $\lambda$. Note that, because $|\lambda|=3n+2-k$, $\ell(\lambda)=2n-k$, and $\lambda$ has at least $n+1$ parts of size at least $2$, $m_\lambda\le n-k$ and the third column of $\lambda$ consists of no more than $1$ cell. Then, $\lambda$ is either $\{2^{n+1}, 1^{n-2-k}\}$ or $\{3, 2^n, 1^{n-k-1}\}$. Either case gives a contradiction since they do not correspond to a numerical semigroup due to the fact that twice of multiplicity is not in the corresponding numerical set. Thus, $\mu$ has no more than $n$ parts, completing the proof that $\mu\in \setB(n)$. Furthermore, $\phi$ is injective because given a partition $\mu$ in the image of $\phi$, the only possible partition it could be the image of is $\{\mu_1+1,\mu_2+1,\ldots \mu_\ell+1,1^{2n-k-\ell}\}$.

Let $\mathcal{E}(n,k)=\setB(n)\setminus \phi(\setGenus(3n+2-k,2n-k)\setminus\{\{n+2,2,1^{2n-k-2}\},\{n+3,1^{2n-k-1}\}\})$. Next, we examine the size of $\mathcal{E}(n,k)$. Let $\pi\in \setB(n)$ and define $\hat{\pi}=\{\pi_1+1,\pi_2+1,\ldots,\pi_\ell+1,1^{2n-k-\ell}\}$. Note that 
\[
    m_{\hat\pi}=2n-k-\ell+1 \label{mhatpi},\quad \text{or, equivalently,}\quad  \ell=2n-k-m_{\hat\pi}+1.
\]
Additionally, since $|\pi|=n+2$, $\pi_1+\ell\le n+3$, so 
\begin{align*}
    \pi_1\le\ & n+3-\ell \label{pi1bound}\\ 
    =\ & 2-n+k+m_{\hat\pi}. \nonumber
\end{align*}
Finally, $m_{\hat\pi}>2n-k-n=n-k$ since $\ell\le n$. 

We proceed with cases based on the size of $m_{\hat\pi}$. The cases show that all partitions $\pi$ in $\mathcal{B}(n)$ map to partitions $\hat\pi$ in $\setGenus(3n+2-k, 2n-k)$, except for the exceptional partitions in the table below.

$$\begin{array}{cl}
 \toprule
 k & \mathcal{E}(n,k) \\ \midrule
 0 & \{4, 1^{n-2}\}, \{2, 2, 1^{n-2}\} \\ 

1 & \{5, 1^{n-3}\}, \{3, 2, 1^{n-3}\}, \{3,1^{n-1}\} \\ 
 
2 & \{6, 1^{n-4}\}, \{4, 2, 1^{n-4}\}, \{2, 2,2, 1^{n-4}\}, \{3, 1^{n-1}\} \\ 
\bottomrule
\end{array}$$

\begin{enumerate}
    \item[Case 1.] If $m_{\hat\pi}>n+2$, then $f_{\hat\pi}=2n-k+\pi_1\le n+2+m_{\hat\pi}<2m_{\hat\pi}$, so $S_{\hat\pi}$ is a numerical semigroup. 
    \item[Case 2.] If $m_{\hat\pi}=n+2$, then either $\pi_1<k+4$ or $\pi=\{k+4,1^{n-k-2}\}$. If $\pi_1<k+4$, then $f_{\hat\pi}<2n+4=2m_{\hat\pi}$, so $S_{\hat\pi}$ is a numerical semigroup. If $\pi=\{k+4,1^{n-k-2}\}$, then $\hat\pi=\{k+5,2^{n-k-2},1^{n+1}\}$, which does not correspond to a numerical semigroup, so $\{k+4,1^{n-k-2}\}\in  \mathcal{E}(n,k)$. 
    \item[Case 3.] If $m_{\hat\pi}=n+1$, then $\pi_1\le k+3$. If $\pi_1<k+2$, then $f_{\hat\pi}<2n+2=2m$, so $S_{\hat\pi}$ is a numerical semigroup. Otherwise, $\pi=\{k+2,2,1^{n-k-2}\}$ or $\pi=\{k+3,1^{n-k-1}\}$. In the first case, $S_{\hat\pi}$ is not a numerical semigroup, so $\{2+k,2,1^{n-k-2}\}\in \mathcal{E}(n,k)$. In the latter case, $f_{\hat\pi}=2n+3$ and $2m=2n+2\in S_{\hat\pi}$. Since $m+1=n+2\not\in S_{\hat\pi}$, $S_{\hat\pi}$ is a numerical semigroup. 
    \item [Case 4] If $m_{\hat\pi}=n$ or $m_{\hat\pi}=n-1$, the possible options for $\pi$ are collected in the following table.
$$   \begin{array}{cl}
\toprule
k & \pi\\
\midrule
1 & \{2,2,1^{n-2}\}, \{3,1^{n-1}\}\\
2 & \{3,2,1^{n-3}\}, \{4,1^{n-2}\}, \{2,2,2,1^{n-4}\}, \{2,2,1^{n-2}\}, \{3,1^{n-1}\}\\
\bottomrule
\end{array}
$$ 
    Each partition can be analyzed individually to determine the exceptional cases. 

\end{enumerate}

Thus, we see that $|\mathcal{E}(n,k)|=k+2$ for $k=0,1,2$, as desired. 
\end{proof}

\begin{cor}
For $n\ge 4$ and $k=0,1,2$, $$\countingFuncGenus(3n+2-k,2n-k)=p(n+2)-6-k.$$
\end{cor}
\begin{proof}
    Note that there are exactly $2$ partitions of $n+2$ with more than $n$ parts ($\{2,1^{n}\}$ and $\{1^{n+2}\}$) and $2$ partitions of $n+2$ with a largest part of size $>n$ ($\{n+2\}$ and $\{n+1,1\}$). Therefore, $b(n)=p(n+2)-4.$
\end{proof}

\begin{remark}\label{remark:genus}
By following a similar structure to the proof of Theorem \ref{GenusBoxedPartitions}, we could find $p(n)-\countingFuncGenus(3n-i,2n-i)$ and $b(n)-\countingFuncGenus(3n+2-i,2n-i)$ for any specific $i$. However, it is an open problem to find functions $h_1(i)$ and $h_2(i)$ such that $p(n)-\countingFuncGenus(3n-i,2n-i)=h_1(i)$ and $b(n)-\countingFuncGenus(3n+2-i,2n-i)=h_2(i)$ for all possible values of $i$. 
\end{remark}

\begin{theorem}\label{genusRecurrence}
Let $\bar{p}(n)$ be the number of partitions of $n$ with no parts of size $1$. For $j\in \N$ and $n\ge j+5$, 
$$\countingFuncGenus(3n+2-j,2n-j)-\countingFuncGenus(3(n-1)+2-j,2(n-1)-j)=\bar{p}(n+2).$$
\end{theorem}

\begin{proof}
Let $\mathcal{\bar{P}}(n)$ be the set of partitions of $n$ with no parts of size $1$. We define a bijection $$\phi:\setGenus(3n+2-j,2n-j)\to \setGenus(3(n-1)+2-j,2(n-1)-j)\cup\mathcal{\bar{P}}(n+2)$$
by  $$\lambda\longmapsto\begin{cases} \lambda \setminus \{2,1\} &\text{if $2\in\lambda$}\\
\{\lambda_1-1,\lambda_2-1,\ldots,\lambda_{2n-j-m_\lambda+1}-1\} &\text{if $2\not\in \lambda$.} 
\end{cases}
$$ 

We recall that nonempty partitions corresponding to numerical semigroups have at least one part of size 1, so this map is well-defined. To show that $\phi$ is, in fact, a bijection, we will show the following four claims:
\begin{enumerate}
    \item For any $\lambda \in \setGenus(3n+2-j,2n-j)$ with $2\not \in \lambda$, $\{\lambda_1-1,\lambda_2-1,\ldots,\lambda_{2n-j-m_\lambda+1}-1\}\in \mathcal{\bar{P}}(n+2)$.
    \item For any $\lambda\in \setGenus(3n+2-j,2n-j)$ with $2\in \lambda$, $\lambda\setminus\{2,1\}\in \setGenus(3(n-1)+2-j,2(n-1)-j)$.
    \item For any $\lambda\in \mathcal{\bar{P}}(n+2)$   $\{\lambda_1+1,\lambda_2+1,\ldots,\lambda_{\ell}+1,1^{2n-j-\ell}\}\in \setGenus(3n+2-j,2n-j)$, where $\ell:=\ell(\lambda)$.
    \item For any $\lambda\in \setGenus(3(n-1)+2-j,2(n-1)-j)$, $\lambda\cup\{2,1\}\in \setGenus(3n+2-j,2n-j)$.
\end{enumerate}

\noindent {\bf Proof of 1:} Let $\lambda\in \setGenus(3n+2-j,2n-j)$ such that $2\not\in \lambda$. Note that, if we subtract $1$ from every part of $\lambda$, the resulting partition has size $3n+2-j-(2n-j)=n+2$ and no parts of size $1$, completing the proof of the first claim. 

\noindent {\bf Proof of 2 and 4:} Let $\lambda$ be a partition of $3n+2-j$ into $2n-j$ parts such that $2,1\in \lambda$. Define $\mu=\lambda\setminus \{2,1\}.$ We will show that $S_\lambda$ is a numerical semigroup if and only if $S_\mu$ is a numerical semigroup. Note that $|\mu|=3n+2-j-3=3(n-1)+2-j$ and $\ell(\mu)=2n-j-2=2(n-1)-j$. Let $k$ be the smallest element in $S_\lambda$ greater than $m_\lambda$, so $\lambda$ has $(2n-j-m_\lambda+1)$ many parts of size at least two and $(2n-j-k+2)$ many parts of size at least three. Similarly, because $k-2$ is the smallest element in $S_\mu$ greater than $m_\mu$, $\mu$ has $(2(n-1)-j-m_\mu+1)$ many parts of size at least two and $(2(n-1)-j-k+4)$ many parts of size at least three. Since the Frobenius number corresponds to the largest hook length, we remove cells inside the largest hook to see that 
 \begin{align*}
F_{\lambda}&\le (3n+2-j)-(2n-j-(m_{\lambda}-1))-(2n-j-k+1)\\
&=m_{\lambda}+k+(j-n)
\end{align*}
and 
\begin{align*}
    F_{\mu}&\le (3n-1-j)-(2(n-1)-j-m_\mu)-(2(n-1)-j-k+3)\\
    &=m_{\mu}+j+k-2-n.
\end{align*}
Thus, since $j<n-4$, we obtain $F_{\lambda}<m_{\lambda}+k$ and $F_{\mu}<m_{\mu}+k-2$. Finally note that, because $m_{\lambda}=m_{\mu}+1$ and $F_{\lambda}=F_{\mu}+2$, $F_\lambda\ne 2m_\lambda\iff F_\mu\ne 2m_\mu$.


\noindent{\bf Proof of 3:} Let $\lambda\in \mathcal{\bar{P}}(n+2)$ and define $\ell:=\ell(\lambda)$. Define $\mu=\{\lambda_1+1,\lambda_2+1,\ldots,\lambda_\ell+1,1^{2n-j-\ell}\}$. Note that $m_{\mu}=2n-j-\ell+1$ and, since $\lambda$ has no parts of size $1$, $\lambda_1\le n+2-2(\ell-1)$. Therefore,
\begin{align*}
   F_{\mu}&=(2n-j)+\lambda_1\\
   &\le (2n-j)+(n+2-2(\ell-1))\\
   &= 3n+4-j-2\ell.
\end{align*}
Thus, we see that $F_\mu-2m_\mu\le 2+j-n<0$, because $j+2<n$. Since $F_\mu<2m_\mu$, $S_\mu$ must be a numerical semigroup. Additionally, $\ell(\mu)=2n-j$ and $|\mu|=|\lambda|+2n-j=3n+2-j$, so $\mu\in \setGenus(3n+2-j,2n-j),$ as desired.


\end{proof}

\begin{remark}
    In the cases where $j\in \{0,1,2\}$, Theorem \ref{genusRecurrence} follows directly from Theorem \ref{GenusBoxedPartitions}. However, in Theorem \ref{genusRecurrence}, we are able to obtain $\countingFuncGenus(3n+2-j,2n-j)-\countingFuncGenus(3(n-1)+2-j,2(n-1)-j)=\bar{p}(n+2)$ for a much wider range of possible $j$s. 
\end{remark}

\section{Partitions of $n$ that correspond to semigroups with Frobenius number $f$}\label{sec:Frobenius}

We begin this section by recalling that for a given partition $\lambda$, the Frobenius number $f_\lambda$ of the corresponding numerical set is equal to the largest hook length of $\lambda$ and that $\countingFuncFrob(n,f)$ denotes the number of partitions of $n$ with Frobenius number $f$. In this section, we collect identities related to $\countingFuncFrob(n,f)$. We begin by counting the number of partitions with exactly one hook that correspond to numerical semigroups. We then expand to counting partitions with fixed difference between the size of the partition and the Frobenius number. 

\begin{theorem}\label{prop:PNf_onehook}
For any positive integers $n$, we have
$$PF(n,n)=\left \lceil \frac{n}{2}\right \rceil.$$
\end{theorem}

\begin{proof}
Let $\lambda$ be a partition of size $n$ that corresponds to the numerical semigroup $S_{\lambda}$ with Frobenius number $n$. Since the Frobenius number of a numerical semigroup is equal to the largest hook length of the corresponding partition, $\lambda$ must consist of exactly one hook, with hook length $n$. To determine the possible number of configurations of that hook, we consider the possible lengths of the leg of the hook, where leg refers to the parts of size $1$. 

We will show that $S_{\lambda}$ is a numerical semigroup if and only if the leg of the hook has length at least $\lfloor\frac{n}{2}\rfloor$. 
First, assume the leg is at least $\lfloor\frac{n}{2}\rfloor$. Since $m_\lambda$ is the smallest nonzero element of $S_\lambda$ and $m_\lambda>\frac{n}{2}$, $x+y>n$ for any $x,y\in S_{\lambda}$. Since $n$ is the Frobenius number of $S_{\lambda}$, we know that $S_\lambda$ is closed and, thus, a numerical semigroup. Next, assume that the leg of the hook has length less than $\lfloor\frac{n}{2}\rfloor$. Then, $\lfloor\frac{n}{2}\rfloor,\lceil\frac{n}{2}\rceil\in S_\lambda$, so $\lfloor\frac{n}{2}\rfloor+\lceil\frac{n}{2}\rceil=n\in S_\lambda$, contradicting the fact that $n$ is the Frobenius number of $S_\lambda$. 

There are $n-\lfloor\frac{n}{2}\rfloor=\lceil\frac{n}{2}\rceil$ many partitions with leg at least $\lfloor\frac{n}{2}\rfloor$, so we see that $\countingFuncFrob(n,n)=\lceil\frac{n}{2}\rceil$.
\end{proof}

Depending on the parity of the size of the partition, we count the number of partitions corresponding to numerical semigroups, where the Frobenius number is close to the size of the partition. 

\begin{theorem} For $n > 2$, we have the following 
 
\begin{enumerate}
\item $\countingFuncFrob(2n,2n-1)=n-1$
\item $\countingFuncFrob(2n+1,2n)=n-2$
\item $\countingFuncFrob(2n,2n-2)=\countingFuncFrob(2n-1,2n-3)=2n-7$
\item $\countingFuncFrob(2n, 2n-3)=3n-11$
\item $\countingFuncFrob(2n-1, 2n-4)=3n-17$
\end{enumerate}
\end{theorem}
\begin{proof}
From Proposition \ref{prop:PNf_onehook}, we have that $\countingFuncFrob(n,n)= \left \lceil \frac{n}{2} \right \rceil$. In each case we consider each partition $\lambda$ in  $\setFrob(n, n-j)$ as a hook of length $n-j$ (exactly those partitions in $\setFrob(n-j,n-j)$) surrounding a partition $\pi$ of size $j$. We then consider all $\left \lceil \frac{n-j}{2} \right \rceil p(j)$ possible combinations of hooks and partitions of size $j$, and note those which do not correspond to numerical semigroups.\\
\noindent{\bf Proof of 1:}  Consider nesting a partition of size $1$ into each of the $n$ partitions in $\setFrob(2n-1, 2n-1)$. This is possible in every case except for the partition $(1^{2n-1})$ (where the resulting diagram would not correspond to a partition). Thus, there are exactly $n-1$ partitions in $\countingFuncFrob(2n,2n-1)$. \\
\noindent{\bf Proof of 2:} Consider nesting a partition of size $1$ into each of the $n$ partitions in $\setFrob(2n, 2n)$. The resulting diagram represents a partition that corresponds to a numerical semigroup with Frobenius number $2n$ in all but two cases: the partition of all ones (where the resulting diagram would not correspond to a partition), and the partition with $n$ ones and one part of size $n$ (where the resulting partition $\lambda$ would have $m_\lambda=n$ and $F_\lambda = 2n$, and would no longer correspond to a numerical semigroup).\\ 
\noindent{\bf Proof of 3:} Consider the two sets of partitions $\setFrob(2n-2,2n-2)$ and $\setFrob(2n-3,2n-3)$, which each have $n-1$ partitions. For each of these, consider the two options for inserting a partition of size 2 inside the largest hook: the map $\phi_1: \lambda\mapsto \phi_1(\lambda)$ adding two to the second part of $\lambda$, or $\phi_2: \lambda\mapsto \phi_2(\lambda)$ adding one to each of the second and third parts of the original partition. This procedure produces $2(n-1)$ diagrams for each set.  In both case there are five diagrams which do not represent partitions corresponding to numerical semigroups. In both cases, the remaining partitions will correspond to numerical semigroups, leaving the size of the set as $2(n-1) -5 = 2n-7$.

\begin{table}[h]
    \centering
$$
\begin{array}{cccc}
\toprule
    \lambda \in \setFrob(2n-2,2n-2)& \phi_1(\lambda) &\phi_2(\lambda) & \text{reason for failing}  \\ \midrule
    \{1^{2n-2}\} & \{1,3,1^{2n-4}\} & \{ 1, 2^2, 1^{2n-5}\} &\text{not a partition}\\
    \{2,1^{2n-4}\} & \{2,3,1^{2n-5}\}& \ & \text{not a partition}\\
    \{n-1,1^{n-1}\} &\{n-1,3, 1^{n-2}\} & \ &2m_{\phi_1(\lambda)} = F\\
    \{n,1^{n-2}\} & \ &\{n,2^2, 1^{n-4}\}  &2m_{\phi_2(\lambda)} = F 
    \\ \bottomrule
    \toprule
    \lambda \in \setFrob(2n-3, 2n-3) & \phi_1(\lambda) &\phi_2(\lambda) & \text{reason for failing}  \\ \midrule
    \{1^{2n-3}\} & \{1,3,1^{2n-5}\} & \{ 1, 2^2, 1^{2n-6}\} &\text{not a partition}\\
    \{2,1^{2n-5}\} & \{2,3,1^{2n-6}\}& \ & \text{not a partition}\\
    \{n-1,1^{n-2}\} &\{n-1,3, 1^{n-3}\} & \{n-1, 2^2, 1^{n-4}\} &m_{\phi_i(\lambda)} + k = F \\
    & \ & \ & \ \ \ \text{for some } k \in S_{\phi_i(\lambda)}
    \\ \bottomrule
    \end{array}
    $$
    \caption{Partitions in $\setFrob(2n-2,2n-2)$ [top] and $\setFrob(2n-3, 2n-3)$ [bottom], whose images under the maps $\phi_1$  and $\phi_2$ are not partitions corresponding to numerical semigroups.}
    \label{}
\end{table}

\noindent{\bf Proof of 4:} Consider nesting partitions of size $3$ into the $n-1$ partitions in $\setFrob(2n-3,2n-3)$. The three possibilities for this addition are represented by: the map $\phi_1:\lambda \rightarrow \phi_1(\lambda)$ adding three to the second part of $\lambda$, the map $\phi_2:\lambda\rightarrow \phi_2(\lambda)$ adding two to the second part and one to the third part of $\lambda$, and the map $\phi_3(\lambda): \lambda\rightarrow \phi_3(\lambda)$ adding one each to the second through fourth parts of $\lambda$. This procedure produces $3(n-1)$ diagrams, and of these, there are 8 diagrams which do not represent partitions corresponding to numerical semigroups.  
\begin{table}[h]
    \centering
$$
\begin{array}{ccccc}
\toprule
    \lambda \in \setFrob(2n-3,2n-3)& \phi_1(\lambda) &\phi_2(\lambda) & \phi_3(\lambda)& \text{reason for failing}  \\ \midrule
    \{1^{2n-3}\} & \{1,4,1^{2n-5}\} & \{ 1, 3, 2, 1^{2n-6}\} & \{1, 2^3, 1^{2n-7}\} &\text{not a partition}\\ \{2, 1^{2n-5}\} &
    \{2,4,1^{2n-6}\} & \{2,3,2,1^{2n-7}\}& \  & \text{not a partition}\\
    \{3,1^{2n-6}\} & \{3, 4, 1^{2n-7}\} & \ &  \   & \text{not a partition}\\
    \{n-1,1^{n-2}\} & \{n-1,4,1^{n-3}\} & \ &\{n-1,2^3, 1^{n-5}\}  & m_{\phi_i(\lambda)} + k = F \\
    & \ & \ & \ &  \ \ \ \text{for some } k \in S_{\phi_i(\lambda)}\\
 \bottomrule
    \end{array}$$
    \end{table}

\noindent{\bf Proof of 5:} Consider nesting the partitions of size $3$ into the $n-2$ partitions in  $\setFrob(2n-4,2n-4)$. Consider again the maps $\phi_1:\lambda \rightarrow \phi_1(\lambda)$ adding three to the second part of $\lambda$, $\phi_2:\lambda\rightarrow \phi_2(\lambda)$ adding two to the second part and one to the third part of $\lambda$, and $\phi_3(\lambda): \lambda\rightarrow \phi_3(\lambda)$ adding one each to the second through fourth parts of $\lambda$. This procedure produces $3(n-2)$ diagrams, and of these, there are 11 diagrams which do not represent partitions corresponding to numerical semigroups.   

\begin{table}[h]
    \centering
$$
\begin{array}{ccccc}
\toprule
    \lambda \in \setFrob(2n-4,2n-4)& \phi_1(\lambda) &\phi_2(\lambda) & \phi_3(\lambda)& \text{reason for failing}  \\ \midrule
    \{1^{2n-4}\} & \{1,4,1^{2n-6}\} & \{ 1, 3, 2, 1^{2n-7}\} & \{1, 2^3, 1^{2n-8}\}&\text{not a partition}\\ \{2, 1^{2n-6}\} &
    \{2,4,1^{2n-7}\} & \{2,3,2,1^{2n-8}\}&  \ & \text{not a partition}\\
    \{3,1^{2n-7}\} & \{3, 4, 1^{2n-8}\} & \ & \   & \text{not a partition}\\
    \{n-2,1^{n-2}\} & \{n-2, 4, 1^{n-3}\} &\{n-2,3,2, 1^{n-4}\} & \{n-2, 2^3, 1^{n-5}\} &m_{\phi_i(\lambda)} + k = F \\
   \{n-3, 1^{n-1}\} & \ & \{n-1, 3, 2, 1^{n-3}\} & \ &  \ \ \ \text{for some } k \in S_{\phi_i(\lambda)}\\
    \{n-4,1^{n}\} & \ & \ &\{n-4,2^3, 1^{n-3}\}  &2m_{\phi_3(\lambda)} = F 
    \\ \bottomrule
    \end{array}$$
    \end{table}
\end{proof}

\begin{remark}\label{remark:Frob} We note that the process described above generalizes. For $j<2n$, any partition $\lambda$ in $\setFrob(2n, 2n-j)$ can be considered as a partition $\pi$ of size $j$ surrounded by a hook of size $2n-j$. There are exactly $\left \lceil \frac{2n-j}{2} \right \rceil$ possible such hooks by Proposition \ref{prop:PNf_onehook}. Through this containment, we have that $\countingFuncFrob(2n,2n-j) \leq p(j)\left \lceil \frac{2n-j}{2} \right \rceil $. This inequality is strict for $j >0$, and in particular we note that 
\[\countingFuncFrob(2n,2n-j)\leq p(j)\left \lceil\frac{2n-j}{2}\right\rceil-\sum_{i=1}^{n-\lfloor\frac{j}{2}\rfloor} p(j,i),\]
where  $p(j,i)$ is the function that counts the number of partitions of $j$ with largest part greater than or equal to $i$. Indeed, the rightmost term removes the forbidden situation where the second row of the Ferrers diagram would be longer than the first row. 
\end{remark}

Next, we obtain a recurrence relation for the number of partitions for which the difference between the size and Frobenius number is constant. 

\begin{theorem}\label{thm:PNf_large}

 For $n\ge 4k+3$,
\[
\countingFuncFrob(n,n-k)-\countingFuncFrob(n-2,n-k-2)=p(k),
\]
where $p(k)$ represents the number of partitions of $k$. 
\end{theorem}

\begin{proof}

Consider the set of partitions $\setFrob(n,n-k)$ and recall the largest hook length of each of these partitions is $n-k$. We consider the set $\mathcal{A}_{n,n-k} \subseteq \setFrob(n,n-k)$ of partitions  $\lambda$ with $\lambda_1>\lambda_2$, and the set $\mathcal{B}_{n,n-k} \subseteq \setFrob(n,n-k)$ of partitions $\lambda$ with $\lambda_1 = \lambda_2$. 

Note that $\setFrob(n,n-k)=\mathcal{A}_{n,n-k} \cup \mathcal{B}_{n,n-k}$ and $\mathcal{A}_{n,n-k} \cap \mathcal{B}_{n,n-k} = \emptyset$. We will show that:
\begin{enumerate}
    \item $|\mathcal{A}_{n,n-k}| = \countingFuncFrob(n-2,n-k-2)$
    \item $|\mathcal{B}_{n,n-k}| = p(k)$
\end{enumerate}

\noindent\textbf{Proof of 1:}
We define a map 
\[ \phi:\setFrob(n-2, n-k-2) \rightarrow  \mathcal{A}_{n,n-k} \]
by 
\[ \{\lambda_1, \lambda_2, \ldots, \lambda_\ell\} \mapsto \{\lambda_1+1,\lambda_2,\ldots,\lambda_\ell,1\}.  \]


First, we will show that  $\phi$ is well-defined. The size of $\phi(\lambda)$ is exactly $n$, and the Frobenius number, $F_{\phi(\lambda)}$ is exactly $n-k$.  

 To show that $S_{\phi(\lambda)}$ is a numerical semigroup, we first need to show that $m_\lambda> k$. Suppose towards a contradiction that $m_\lambda \leq k$. We know that the largest hook length of $\lambda$ is $n-k-2$, and the number of boxes in the Ferrers diagram of $\lambda$ not included in the largest hook is exactly $k$. Note that the largest hook length of $\lambda$ is equal to $\lambda_1-\lambda_2 + h_{(2, 1)}(\lambda) + 1$, where $h_{(2,1)}$ is the hook length of cell $(2,1)$ of $\lambda$, and $\lambda_{i} -\lambda_{i+1} < m_\lambda$ for all $1<i<\ell$. Then, since the hook $h_{(2,1)}(\lambda)$ includes all parts of size one, and beyond that at most $k$ additional boxes, we have  $h_{(2,1)}(\lambda) \leq m_\lambda + k$, so
\[n-k-2 =\lambda_1-\lambda_2 + h_{(2, 1)}(\lambda) + 1  \leq \lambda_1-\lambda_2 + m_\lambda + k + 1 < 2m_\lambda + k + 1 \leq 3k +1.\]
Since we assumed that $n \geq 4k+3$, we get a contradiction. So $m_\lambda> k$.

We now show that $S_{\phi(\lambda)}$ is a numerical semigroup. First, we have $m_{\phi(\lambda)} = m_{\lambda}+1$ and $f_{\phi(\lambda)} = f_\lambda + 2$. Note that $h_{(2,1)}(\lambda) \leq m_\lambda + k < 2m_\lambda$, so $h_{(2,1)}(\phi(\lambda)) < 2m_\lambda +1$. Thus any elements of the complement of $S_{\phi(\lambda)}$ are either equal to $f_{\phi(\lambda)}$ or are less than $2m_\lambda +1 < 2m_{\phi(\lambda)}$. To show that $S_{\phi(\lambda)}$ is closed under addition, let $x, y$ be nonzero elements of  $S_{\phi(\lambda)}$ (if $x$ or $y$ were zero, then $x+y \in S_{\phi(\lambda)}$). If $x>f_{\phi(\lambda)}-1$ or $y>f_{\phi(\lambda)}-1$, then since $x,y>1$, we have $x+y>f_{\phi(\lambda)}$. Now we assume that both $x$ and $y$ are less than $f_{\phi(\lambda)} -1$. Thus, there exist $x', y' \in S_{\lambda}$ such that $x = x'+1$ and $y = y'+1$. Note that $x', y' \geq m_{\lambda}$. Then, $x+y = x'+y'+2 > 2m_\lambda+1$. Furthermore, $x'+y'$ cannot equal $f_\lambda$, so $x+y\ne f_\lambda+2=f_{\phi(\lambda)}$.  Since all elements of the complement of $S_{\phi(\lambda)}$ are either less than $2m_{\lambda}+1$ or $f_{\phi(\lambda)}$, $S_{\phi(\lambda)}$ is closed under addition, and so is a numerical semigroup. Thus, $\phi(\lambda) \in \mathcal{A}_{n, n-k}$.

Next, we consider the inverse map $\phi^{-1}$ for $\mu \in \mathcal{A}_{n,n-k}$. We obtain $\phi^{-1}(\mu) = \{\mu_1-1, \mu_2, \ldots, \mu_{\ell-1}\}$, and note that the removed part $\mu_r =1$. Thus $\phi^{-1}(\mu)$ is a map from $\mathcal{A}_{n,n-k}$ to partitions $\phi^{-1}(\mu)$ of $n-2$ with Frobenius number $n-k-2$.

We will show that the numerical sets that correspond to these partitions are in fact numerical semigroups. Note that $m_{\phi^{-1}(\mu)} = m_\mu - 1$ and $f_{\phi^{-1}(\mu)} = f_\mu -2$. Similarly to above, $h_{(2,1)}(\mu) \leq m_\mu+k$, so $h_{(2,1)}(\phi^{-1}(\mu)) = h_{(2,1)}(\mu)-1 \leq m_\mu+k-1\leq m_{\phi^{-1}(\mu)}+k$.
Then $n-k=h_{(2,1)}(\mu)+ \mu_1 -\mu_2 +1 \leq m_\mu + k + \mu_1 - \mu_2 +1$. 
Suppose toward a contradiction that $m_\mu \leq k+1$, then, $n-k \leq 2k +1 + \mu_1-\mu_2 +1  $, so $n-3k-2 \leq \mu_1 -\mu_2$. Because $n$ must be at least $4k+3$, $\mu_1 - \mu_2 \geq k+1 \geq m_\mu$. This contradicts the fact that $S_\mu$ is a numerical semigroup, so $m_\mu > k+1$. 

We have
$h_{(2,1)}(\mu) < 2m_\mu -1$, so $h_{(2,1)}(\phi^{-1}(\mu)) < 2m_\mu -2 = 2m_{\phi^{-1}(\mu)}.$ 
 Similar to above, all elements of the complement of $S_{\phi^{-1}(\mu)}$ are either less than $2m_{\phi^{-1}(\mu)}$, or are $f_{\phi^{-1}(\mu)}$. Let $x, y$ be nonzero elements of $S_{\phi^{-1}(\mu)}$. Either both $x$ and $y$ are less than $f_{\phi^{-1}(\mu)}$, or at least one of them is greater than $f_{\phi^{-1}(\mu)}$. We can see that $x+y$ must be in $S_{\phi^{-1}(\mu)}$. Thus, $S_\phi^{-1}(\mu)$ is closed under addition, and is a numerical semigroup. Thus, since $\phi$ is a bijection, we have that $|\mathcal{A}_{n,n-k}| = \countingFuncFrob(n-2,n-k-2)$.

\noindent \textbf{Proof of 2:} We now show that $|\mathcal{B}_{n,n-k}| = p(k)$.  Let $\lambda \in \mathcal{B}_{n,n-k}$. Define a map 
\[\psi: \mathcal{B}_{n,n-k} \rightarrow \mathcal{P}(k)\] by
\[\{\lambda_1, \lambda_2, \ldots, \lambda_\ell\} \mapsto \{ \lambda_2 -1, \lambda_3 -1, \ldots, \lambda_\ell -1\}\] Note that this map removes the largest hook of $\lambda$ to get $\psi(\lambda)$, a partition of size $k$. Since $F_\lambda=n-k$, we have $\lambda_1=n-k-(\ell-1)$. Thus, $|\psi(\lambda)|=\sum\limits_{i=2}^{\ell}\lambda_i-(\ell-1)=n-(n-k-(\ell-1))-(\ell-1)=k$. What remains to show is that $\phi$ is surjective. We will show that $\psi^{-1}$ is well-defined. We get the partition $\psi^{-1}(\mu)$ from a partition $\mu$ of $k$ by adding a hook of length $n-k$ with arm exactly equal to $\mu_1$. We now show that $\psi^{-1}(\mu)$ corresponds to a numerical semigroup. Since $\psi^{-1}(\mu)_1 = \psi^{-1}(\mu)_2$, we have,

\[ n-k = h_{(2,1)}(\psi^{-1}(\mu)) + 1\leq m_{\psi^{-1}(\mu)} + k +1.\]

Rewriting the inequality for $m_{\psi^{-1}(\mu)}$ we have
\[ m_{\psi^{-1}(\mu)} \geq n-2k-1 \geq 2k +2.\]
Because $n \geq 4k+3$, $n-k < n - 1 \leq n + (n - 4k -2) = 2n-4k-2$. So $f_{\psi^{-1}(\mu)} < 2n-4k-2 = 2(n-2k-1) \leq 2m_{\psi^{-1}(\mu)}$. Thus, $S_{\psi^{-1}(\mu)}$ is a numerical semigroup and $\psi$ is a bijection.

\end{proof}

\section{Partitions of $n$ that correspond to semigroups with multiplicity $m$} \label{sec:multiplicity}

We begin this section by recalling that for a given partition $\lambda$, the multiplicity $m_\lambda$ of the corresponding numerical set is one more than the number of parts of size one in $\lambda$ and $\countingFuncMult(n,m)$ denotes the number of partitions of $n$ corresponding to a numerical semigroup with multiplicity $m$. Thus, by considering the partition that remains after removing all parts of size one, we illustrate the connection between $\countingFuncMult(n,m)$ and $\bar{p}(n-m+1)$, where $\bar{p}(n)$ counts the number of partitions of $n$ with no parts of size one. We leverage this relationship to further investigate the partitions in $\setMult(n,m)$ which are equinumerous with particular subsets of partitions counted by $\overline{p}(n)$.

\begin{prop}\label{prop:MultPartitionNumbers}
For  all positive integers $m$ and $n$, we have $\countingFuncMult(n,m)\le \overline{p}(n+1-m)$. Furthermore, when $m> \frac{n}{2}$, $\countingFuncMult(n,m)= \overline{p}(n+1-m)$.
\end{prop}

\begin{proof}
We begin by noting that, $$p(n\mid \text{exactly $m$ parts of size $1$})=\bar{p}(n-m+1).$$ Since $\countingFuncMult(n,m)\le p(n\mid \text{exactly $m$ parts of size $1$})$, the desired inequality follows.

Next, let $m\ge\frac{n}{2}$ and let $\lambda$ be a partition of $n$ with exactly $m-1$ parts of size 1. Thus $m_\lambda=m$ Since $m_\lambda> \frac{n}{2}$, $2m_\lambda> n$. Furthermore, since $n$ is the size of the partition, $F_\lambda\le n$. Thus, 
$2m_\lambda>F_\lambda$, proving that $S_\lambda$ is a numerical semigroup. 
\end{proof}

For the next results, we introduce new subsets of partitions counted by $\bar{p}(n)$ to explore the exact values of $\countingFuncMult(n,m)$ when $m\approx \frac{n}{2}$, but $m\le \frac{n}{2}$. Let $\mathcal{A}(n)$ be the set of partitions of $n$ with at least two parts and no parts of size $1$ and  $a(n):=|\mathcal{A}(n)|$.

\begin{theorem}\label{prop:multAlmostHalf}
    For $n\ge 5$,
\begin{align}
\countingFuncMult(2n,n)&=a(n+1),\label{PNm2n}\\
\countingFuncMult(2n-1,n-1)&=a(n+1)-1,\label{PNm2n-1}\\
\countingFuncMult(2n-2,n-2)&=a(n+1)-2,\label{PNm2n-2}\\
\countingFuncMult(2n-3,n-3)&=a(n+1)-5.\label{PNm2n-3}
\end{align}
\end{theorem}
\begin{proof}
For each $i\in \{0,1,2,3\}$, we define an injection $$\phi_i:\setMult(2n-i,n-i)\to \mathcal{A}(n+1)$$ by 
$$\lambda\mapsto\lambda\setminus\{1^{n-i-1}\}.$$ Note that, since we can define an inverse map by adding $n-i-1$ parts of size $1$ to a partition in $a(n+1)$, we know that $\phi_i$ is injective.  To prove the proposition, we must prove the following claims:
\begin{enumerate}
    \item For any $i\in \{0,1,2,3\}$ and $\lambda\in \setMult(2n-i,n-i)$, $\phi_i(\lambda)\in \mathcal{A}(n+1)$.
    \item For $i\in \{0,1,2\}$, $|\phi_i(\lambda)|=a(n+1)-i$, and, for $i=3$, $|\phi_i(\lambda)|=a(n+1)-5$.
\end{enumerate}

\noindent{\bf Proof of 1:} Fix $i\in \{0,1,2,3\}$ and let $\lambda\in \setMult(2n-i,n-i)$. Since $m_\lambda=n-i$, $\lambda\setminus\{1^{n-i-1}\}$ has no parts of size 1. Furthermore, if the resulting partition is only one part, then $\lambda=\{n+1,1^{n-i-1}\}$, and $n-i, n\in S_\lambda$, but $2n-i\not\in S_\lambda$, contradicting the fact that $S_\lambda$ is a numerical semigroup.  Therefore, $\phi_i(\lambda)\in \mathcal{A}(n+1)$.\\
\noindent{\bf Proof of 2:} Fix $i\in \{0,1,2,3\}$. We will show that, except for $i$ exceptional cases (or 5 exceptional cases when $i =3$), if we take a partition $\mu\in\mathcal{A}(n+1)$, then $\phi_i^{-1}(\mu)=\mu\cup\{1^{n-i-1}\}$ corresponds to a numerical semigroup. Let $\mu\in \mathcal{A}(n+1)$ and define $\lambda=\phi_i^{-1}(\mu)$. Then, $f_\lambda=f_\mu+n-i-1\le n+n-i-1<2n-i$. Let $x,y\in S_\lambda$ with $x,y\ne 0$. Then $x+y\ge 2n-2i$. Thus, if $S_\lambda$ is not a numerical semigroup, we must have $f_\lambda\ge 2n-2i$. In Table \ref{tab:Aexceptions}, we collect the possible partitions $\lambda\in \phi_i^{-1}(\mathcal{A}(n+1))$ with $f_\lambda\ge 2n-2i $ and evaluate when $S_\lambda$ is a numerical semigroup. 
\begin{table}[h]
    \centering
    $$
\begin{array}{ccc}
\toprule
    i &  \{\lambda\in \phi_i^{-1}(\mathcal{A}(n+1))\mid f_\lambda\ge 2n-2i\} &\begin{minipage}{1.75in} \centering {Number of exceptional cases} \end{minipage}\\ \midrule
    0 & \emptyset & 0 \\
    1 & \{n-1,2,1^{n-2}\} & 1 \\
   2 &\{\{n-2,3,1^{n-3}\},\{n-3,2,2,1^{n-3}\}, \{n-1,2,1^{n-3}\}^*\} & 2 \\
    & \multirow{2}{3in}{
       $\{\{n-1,2,1^{n-4}\}, \{n-2,3,1^{n-4}\}, \{n-3,4,1^{n-4}\},$ $\{n-4,3,2,1^{n-4}\}, \{n-5,2,2,2,1^{n-4}\}\}$} & \\
   3 && 5 \\ \\ \bottomrule
\end{array}
$$
    \caption{The collected cases where $f_\lambda \geq 2n-2i$ for the proof of Proposition \ref{prop:multAlmostHalf}. The partition marked $*$ does correspond to a numerical semigroup, so is not counted as an exceptional case.}
    \label{tab:Aexceptions}
\end{table}
\end{proof}

\begin{remark} \label{remark:mult}
By following a similar structure, we could find $a(n+1)-\countingFuncMult(2n-i,n-i)$ for any specific $i$. However, it is an open problem to find a specific function $h(i)$ such that $a(n+1)-\countingFuncMult(2n-i,n-i)=h(i)$ for all possible values of $i$. 
\end{remark}




 \begin{lemma}\label{Lemma:multInjection}
For positive integers $j$ and $n$ such that $n\le 3j$, the map \[\phi:\setMult(n,j)\rightarrow\setMult(n+2,j+1\mid\text{first two parts distinct})\] 
defined by 
\[
\mu=\{\mu_1,\mu_2, \ldots, \mu_k, 1^{j-1}\}\mapsto\{\mu_1+1, \mu_2, \ldots, \mu_k,1^j\}
\]
is bijective.
\end{lemma}

\begin{proof}
Fix integers $j,n$ such that $n\le 3j$. Define the map \[\phi:\setMult(n,j)\rightarrow\setMult(n+2,j+1\mid\text{first two parts distinct})\] 
 by 
\[
\mu=\{\mu_1,\mu_2, \ldots, \mu_k, 1^{j-1}\}\mapsto\{\mu_1+1, \mu_2, \ldots, \mu_k,1^j\}.
\] Note that if $\lambda=\{\lambda_1,\lambda_2,\ldots,\lambda_j,1^j\}\vdash n$ with $\lambda_1\ne \lambda_2$, we can define $\phi^{-1}(\lambda)=\{\lambda_1-1,\lambda_2,\ldots,\lambda_j,1^{j-1}\}$. Thus, the bijectivity of $\phi$ as a map between sets of partitions is straightforward.

Define $\mathcal{P}_a(n,j)$ to be the set of partitions of size $n$ with $j-1$ parts of size 1 and $\mathcal{P}_b(n,j)$ to be the subset of $\mathcal{P}_a(n,j)$ where the first two parts of the partition are distinct. Let $\mu \in \mathcal{P}_a(n,j) $ and $ \lambda \in \mathcal{P}_b(n+2,j+1)$ such that $\phi(\mu)=\lambda$. We will show that $S_\lambda$ is a numerical semigroup if and only if $S_\mu$ is a numerical semigroup.

Note that $|\lambda|=|\mu|+2$, $m_\lambda=j=m_\mu+1$, and $F_\lambda=F_\mu+2$. Thus, 
\[
F_\mu<2m_\mu \iff F_\lambda<2m_\lambda.
\]
Since all partitions $\mu$ such that $F_\mu<2m_\mu$ correspond to numerical semigroups,  we see that when $F_\mu<2m_\mu$ or $F_\lambda<2m_\lambda$, both $\mu$ and $\lambda$ are partitions corresponding to numerical semigroups. 

Thus, we now assume that $F_\mu>2m_\mu$ and $F_\lambda>2m_\lambda$. Since $F_\mu\le |\mu|=n$ and $n\le 3j=3m_\mu$, we know that $F_\mu\le 3m_\mu$. Similarly, $F_\lambda\le 3m_\lambda-1$. 

Let $k$ be the number of parts of size larger than one in $\mu$. By the definition of $\phi$, $k$ is also the number of parts of size larger than one in $\lambda$.  Thus, because
\[
F_{\mu}=(m_{\mu}-1)+k+\mu_1-1>2m_{\mu},
\]
 we know $k+\mu_1-1>m_{\mu}+1$.
Define $\pi$ to be the partition obtained after removing the largest hook from $\mu$ (or, equivalently, from $\lambda$). Then,
\begin{align*}
    |\pi|&=|\mu|-F_{\mu}\\
    &=n-(m_\mu-1)-k-(\mu_1-1)\\\
    &<n+1-m_\mu-(m_\mu+1)\\
    &=n-2m_\mu \le j=m_\mu.
\end{align*}
Consider $$\max\left(\N\setminus(S_\mu\cup\{F_\mu\})\right)\le m_\mu+|\pi|<2m_\mu$$
and
$$\max\left(\N\setminus(S_\lambda\cup\{F_\lambda\})\right)\le m_\lambda+|\pi|<2m_\lambda.$$ Thus, $F_\mu$ (resp. $F_\lambda$) is the only element of  $\N\setminus S_\mu$ (resp. $\N \setminus S_\lambda$) that is $\ge 2m_\mu$ (resp. $2m_\lambda$). 
Finally, note that $a,b\in S_\mu$ with $a+b=F_\mu$ if and only if $a+1,b+1\in S_\lambda$ with $a+b+2=F_\lambda$, meaning that $S_\mu$ is not a numerical semigroup if and only if $S_\lambda$ is not a numerical semigroup, completing our proof.
\end{proof}

Let $\bar{A}(n)$ be the set of partitions of $n$ with at least two parts with no parts of size $1$ and where $\lambda_1=\lambda_2$ and $\bar{a}(n):=|\bar{A}(n)|$. 
 
\begin{theorem}\label{Thm:multRec1}
For $j\in \N$ and $n\ge \text 2j+2$, 
$$\countingFuncMult(2n-j,n-j)-\countingFuncMult(2(n-1)-j,(n-1)-j)=\noOnesFunc(n+1).$$
\end{theorem}

\begin{proof}
We define a bijection 
$$\phi:\setMult(2n-j,n-j)\to \setMult(2(n-1)-j,(n-1)-j)\cup \noOnesSet(n+1)$$
by 
$$\lambda\mapsto \begin{cases}\lambda\setminus\{1^{n-j-1}\} &\text{if $\lambda_1=\lambda_2$}\\
\{\lambda_1-1,\lambda_2,\lambda_3,\ldots, \lambda_k,1^{n-j-2}\} &\text{if $\lambda_1>\lambda_2$.}
\end{cases}$$
 
To show that $\phi$ is, in fact, a bijection, we will show the following four claims:
\begin{enumerate}
    \item For any $\lambda \in \setMult(2n-j,n-j)$ with $\lambda_1 = \lambda_2$, $\lambda \setminus \{1^{n-j-1}\} \in \noOnesSet(n+1)$.
    \item For any $\lambda \in\setMult(2n-j,n-j)$ with $\lambda_1 > \lambda_2$, $\{\lambda_1-1,\lambda_2,\lambda_3,\ldots, \lambda_k,1^{n-j-2}\} \in \setMult(2(n-1)-j,(n-1)-j).$
    \item For any $\mu \in \noOnesSet(n+1)$, $\mu \cup  \{1^{n-j-1}\}  \in \setMult(2n-j,n-j)$.
    \item For any $\mu \in \setMult(2(n-1)-j,(n-1)-j),$ $\{\mu_1 +1, \mu_2, \mu_3, \ldots, \mu_k, 1^{n-j-1}\}\in \setMult(2n-j,n-j).$
\end{enumerate}
Note that, Claims 1 and 2 show that $\phi$ is well-defined. Then, we notice that, if $\mu\in \noOnesSet(n+1)$, $\phi(\mu\cup\{1^{n-j-1}\})=\mu$. Furthermore, if $\mu=\{\mu_1,\mu_2,\ldots,\mu_k,1^{n-j-2}\}\in \setMult(2(n-1)-j,(n-1)-j)$, then $\phi(\{\mu_1+1,\mu_2,\ldots,\mu_k,1^{n-j-1}\})=\mu$. Therefore, Claims 3 and 4 show that $\phi^{-1}$ is well-defined.

To prove Claims 1-4, we begin by noting that all $\lambda \in\setMult(2n-j,n-j)$ have exactly $n-j-1$ parts of size 1, so we can write such $\lambda$ in the form $\{\lambda_1, \lambda_2, \ldots, \lambda_k, 1^{n-j-1}\}$. 

\noindent {\bf Proof of 1:} Let $\lambda \in \setMult(2n-j,n-j)$ with $\lambda_1 = \lambda_2$. Then, $\mu=\lambda \setminus \{1^{n-j-1}\}$ has no parts of size $1$ and $\mu_1=\mu_2$, so $\lambda \setminus \{1^{n-j-1}\} \in \noOnesSet(n+1)$.

\noindent {\bf Proof of 2:} Note that, when $n=2j+2$, $2(n-1)-j=3j$ and $3((n-1)-j)=3j+4$. Furthermore, if we increase $n$ by $1$, $2(n-1)-j$ increases by $2$ and $3((n-1)-j)$ increases by $3$, so $2(n-1)-j\le 3((n-1)-j)$ for any $n\ge 2j+2$. Therefore, this claim follows directly from Lemma \ref{Lemma:multInjection}. 

\noindent {\bf Proof of 3:} Let $\lambda\in\mathcal{\bar{A}}(n+1)$ and let $\ell$ be the length of $\lambda$. By considering the largest hook length of $\lambda$, we get
\[
\lambda_1+\ell -1\le (n+1)-(\lambda_1-1)-(\ell-2),
\]
which gives $\lambda_1+\ell\le\frac{1}{2}(n+5)$. Consider  $\mu=\lambda\cup\{1^{n-j-1}\}$, and note that $m_{\mu} = n-j$.
Using the definition of hook length, we get
\begin{align*}
   F_{\mu}&=(\lambda_1-1)+\ell+(n-j-1) \\
   &\le \frac{1}{2}n+(n-j)+1<2m_{\mu},
\end{align*}
where the last inequality comes from the fact that $j\le \frac{1}{2}n-1$.

\noindent {\bf Proof of 4:} This claim follows from Lemma \ref{Lemma:multInjection}. 
\end{proof}
By combining Theorem \ref{prop:multAlmostHalf} with the case $j=0$ of Theorem \ref{Thm:multRec1} and separately confirming the cases with $1\le n <5$, we obtain the following identity relating $a(n)$ and $\overline{a}(n)$.
\begin{cor}\label{cor:aabaridentity}
    For $n\in \N$, $$a(n+1)-a(n)=\overline{a}(n+1).$$
\end{cor}
\noindent Note that Corollary \ref{cor:aabaridentity} can also be proved combinatorially by noting that the partitions in $\mathcal{A}(n+1)\setminus\noOnesSet(n+1)$ are exactly those where it is possible to subtract one from the largest part to obtain a partition in $\mathcal{A}(n)$. 

We now extend the recurrence from Theorem \ref{Thm:multRec1} to the case where $n = 2j+1$, and subsequently the case $n = 2j$. The two results combined show that the restriction on $n$ in Theorem \ref{Thm:multRec1} is the best possible. 

 \begin{theorem}
For $j\geq 4$, 
$$\countingFuncMult(3j+2,j+1)-\countingFuncMult(3j,j)=\noOnesFunc(2j+2)-j.$$
\end{theorem}

\begin{proof}
Consider the set $\mathcal{A}= \noOnesSet(2j+2) \setminus \mathcal{I}$, where $\mathcal{I} = \left \{ \{s,s, 2^{j-s+1}\}\ |\ 2\leq s\leq j+1\right \}$. Note that 
$|\mathcal{I}|=j$, so our goal is to prove that $\setMult(3j+2,j+1)$ is equinumerous with $\setMult(3j,j)\cup \mathcal{A}$.  To that end, we will prove that the map
$$\phi:\setMult(3j+2, j+1) \to \setMult(3j,j)\cup \mathcal{A}$$
defined by 
$$\lambda\mapsto \begin{cases}\lambda\setminus\{1^{j}\} &\text{if $\lambda_1=\lambda_2$}\\
\{\lambda_1-1,\lambda_2,\lambda_3,\ldots, \lambda_k, 1^{j-1}\} &\text{if $\lambda_1>\lambda_2$}
\end{cases}$$ is a bijection.
To show that $\phi$ is, in fact, a bijection, we will show the following four claims:
\begin{enumerate}
    \item For $\lambda\in\setMult(3j+2,j+1)$ with $\lambda_1=\lambda_2$, $\lambda\setminus\{1^{j}\}\in \mathcal{A}$.
    \item For $\lambda\in\setMult(3j+2,j+1)$ with $\lambda_1>\lambda_2$, $\{\lambda_1-1,\lambda_2,\lambda_3,\ldots, \lambda_k, 1^{j-1}\}\in\setMult(3j,j)$.
    \item For $\lambda\in \mathcal{A}$, $\lambda\cup\{1^{j}\}\in\setMult(3j+2,j+1)$ with $\lambda_1=\lambda_2$.
    \item For $\lambda\in \setMult(3j,j)$, $\{\lambda_1+1,\lambda_2,\lambda_3,\ldots, \lambda_k, 1^{j-1}\}\in\setMult(3j+2,j+1)$.
\end{enumerate}
Note that, for $2\le s\le j+1$,  $\lambda=\{s,s,2^{j-s+1}, 1^{j}\}\not\in\setMult(3j+2,j+1)$ $m_\lambda=j+1$ and $F_\lambda=2(j+1)$. Claims 1 and 2 show that $\phi$ is well-defined and claims 3 and 4 show that $\phi^{-1}$ is well-defined. 

\noindent {\bf Proof of 1:} Let $\lambda\in\setMult(3j+2,j+1)$ with $\lambda_1=\lambda_2$. Then, $\lambda\setminus\{1^j\}$ has no part of size $1$, is of size $(3j+2)-j=2j+2$, and the first two parts of $\lambda$ are the same, so 
$\lambda\setminus\{1^{j}\}\in \mathcal{A}$.

\noindent {\bf Proof of 2:} This claim follows from Lemma \ref{Lemma:multInjection}.

\noindent {\bf Proof of 3:} Let $\lambda=(\lambda_1,\lambda_2,\ldots,\lambda_k)\in \mathcal{A}$ with $\lambda_1=\lambda_2$. Let $\mu=\lambda\cup\{1^j\}$. Then, we have $|\mu|=|\lambda|+j=(2j+2)+j=3j+2$ and $m_\mu=j+1$. Since $\lambda_1=\lambda_2$ and $\lambda_i\ge 2$ for all $i$, we get $|\lambda|\ge 2(\lambda_1-1)+2k$, which implies that $\lambda_1-1+k\le j+1$. Therefore,
\[
F_\mu=j+k+(\lambda_1-1)\le 2j+1=2m_\mu-1,
\]
so $S_\mu$ is a numerical semigroup.

\noindent {\bf Proof of 4:} This claim follows from Lemma \ref{Lemma:multInjection}.
 \end{proof}
 
 For the proof below, the statement does not meet the conditions of Lemma \ref{Lemma:multInjection}. However, the argument in the third and sixth claims is very similar to that of Lemma \ref{Lemma:multInjection}. 
\begin{theorem}
For $j\geq 4$, 
$$\countingFuncMult(3j,j)-\countingFuncMult(3j-2,j-1)=\noOnesFunc(2j+1)+1.$$
\end{theorem}

\begin{proof}
Consider the set $\mathcal{M}= \setMult(3j,j) \setminus (3, 2^{j-1}, 1^{j-1})$. We note that, because all $\lambda \in\mathcal{M} $ have corresponding numerical set with multiplicity $j$, they each have exactly $j-1$ parts of size 1. Thus, we can write each $\lambda\in \mathcal{M}$ in the form $\{\lambda_1, \lambda_2, \ldots, \lambda_k, 1^{j-1}\}$. We define a map
$$\phi:\mathcal{M} \to \setMult(3j-2,j-1)\cup \noOnesSet(2j+1)$$
by 
$$\lambda\mapsto \begin{cases}\lambda\setminus\{1^{j-1}\} &\text{if $\lambda_1=\lambda_2$}\\
 \{\lambda_1-1,\lambda_1-1,3,2^{j-\lambda_1}\} &\text{if $\lambda=\{\lambda_1,\lambda_1-1,2^{j+1-\lambda_1},1^{j-1}\}$.}\\
 \{ \lambda_1-1, \lambda_2, \lambda_3, \ldots, \lambda_k, 1^{j-2}\} & \text{otherwise}.
 \end{cases}$$
To show that $\phi$ is a bijection, we will show the following six claims:
\begin{enumerate}
    \item For $\lambda \in \mathcal{M}$ with $\lambda_1=\lambda_2$, $\lambda\setminus\{1^{j-1}\}\in\noOnesSet(2j+1)$.
    \item For $\lambda=\{\lambda_1,\lambda_1-1,2^{j+1-\lambda_1}, 1^{j-1}\} \in \mathcal{M}$, $\{\lambda_1-1,\lambda_1-1,3,2^{j-\lambda_1}\}\in\noOnesSet(2j+1)$ and $\{\lambda_1-1,\lambda_1-1,3,2^{j-\lambda_1}\}\ne \tilde\lambda\setminus\{1^{j-1}\}$ for any $\tilde\lambda\in \mathcal{M}$. 
    \item For $\lambda \in \mathcal{M}$ with $\lambda_1>\lambda_2+1$ or $\lambda_1=\lambda_2+1$ and $\lambda_3\neq 2$, $\{ \lambda_1-1, \lambda_2, \lambda_3, \ldots,\lambda_k, 1^{j-2}\}\in\setMult(3j-2,j-1)$.
    \item For $\mu\in\noOnesSet(2j+1)$ such that $F_\mu\ne j+1$,
    $\mu\cup\{1^{j-1}\}\in\mathcal{M}$.
    \item For $\mu\in \noOnesSet(2j+1)$ with $F_\mu=j+1$,  $\{\mu_1+1,\mu_2,\mu_3-1,\ldots,\mu_k,1^{j-1}\}\in\mathcal{M}$. 
    \item For $\mu\in\setMult(3j-2,j-1)$, $\{\mu_1+1, \mu_2, \mu_3, \ldots, \mu_k, 1^{j-1}\}\in\mathcal{M}$ with $\mu_1>\mu_2$ or $\mu_1=\mu_2$ and $\mu_3\neq 2$.
\end{enumerate}

\noindent {\bf Proof of 1:} Let $\lambda\in \mathcal{M}$ such that $\lambda_1=\lambda_2$. Since $\lambda$ has exactly $j-1$ parts of size $1$, $\lambda\setminus\{1^{j-1}\}$ has no parts of size $1$. Furthermore, since $\lambda_1=\lambda_2$, $\lambda\setminus\{1^{j-1}\}$ has at least two appearances of its largest part, proving that $\lambda\setminus\{1^{j-1}\}\in\noOnesSet(2j+1)$.

\noindent {\bf Proof of 2:} Let $\lambda=\{\lambda_1,\lambda_1-1,2^{j+1-\lambda_1}, 1^{j-1}\} \in \mathcal{M}$. Note that, the only partition in $\setMult(3j,j)$ with $\lambda_2=2$ and $\lambda_1=\lambda_2+1=3$ is $\{3,2^{j-1},1^{j-1}\}$, so we may assume that $\lambda_1-1=\lambda_2>2$. Therefore $\{\lambda_1-1,\lambda_1-1,3,2^{j-\lambda_1}\}$ is a partition with a repeated largest part and no parts of size $1$. Furthermore, since $|\lambda|=3j$ and $(\lambda_1-1)+(\lambda_1-1)+3+2\cdot(j-\lambda_1)=|\lambda|-(j-1)$, $|\{\lambda_1-1,\lambda_1-1,3,2^{j-\lambda_1}\}|=2j+1$, making $\{\lambda_1-1,\lambda_1-1,3,2^{j-\lambda_1}\}\in\noOnesSet(2j+1)$. 

It
 remains to show that $\{\lambda_1-1,\lambda_1-1,\lambda_3,2^{j-\lambda_1},1^{j-1}\}\not\in \mathcal{M}$. Define $\tilde{\lambda}=\{\lambda_1-1,\lambda_1-1,\lambda_3,2^{j-\lambda_1},1^{j-1}\}$ and note that $F_{\tilde{\lambda}}=\lambda_1-1+2+(j-\lambda_1)+j-1=2j$. However, $m_{\tilde{\lambda}}=j$, so $S_{\tilde{\lambda}}$ is not closed under addition and, thus, $\tilde\lambda\not\in\mathcal{M}$, as desired. 

\noindent {\bf Proof of 3:}
 Let $\lambda\in \mathcal{M}$ such that $\lambda_1> \lambda_2+1$ or $\lambda_1=\lambda_2+1$ and $\lambda_3\neq 2$. 
Define $\mu=\{\lambda_1-1,\lambda_2,\lambda_3,\ldots,\lambda_{k}, 1^{j-2}\}$. We will show that $S_{\mu}$ is a numerical semigroup. 

First, note that because $F_\lambda=F_\mu+2$ and $m_\lambda=m_\mu+1$, $$F_\lambda<2m_\lambda \iff F_\mu<2m_\mu.$$ Thus, we can focus on the case where $F_\lambda>2m_\lambda.$ Define $\pi$ to be the partition obtained by removing the largest hook from $\lambda$ (or, equivalently, from $\mu$). Then 
\begin{align*}
    |\pi|&=|\lambda|-F_{\lambda}\\
    &<|\lambda|-2m_\lambda\\
    &=j. 
\end{align*}
Then, since $F_\pi\le |\pi|$, we have that $F_\pi<j$.
Furthermore, note that, if $\lambda_3>2$, we obtain the stronger restriction that $F_\pi<j-1$. 

Next, by considering the penultimate vertical step in the diagram of $\mu$, note that
\begin{align*}
    \max(\N\setminus (S_\mu\cup \{F_\mu\}))&=m_\mu+F_\pi\\
    &<m_\mu+j\\
    &=2m_\mu+1.
\end{align*}
Therefore, the only possible element of $\N\setminus S_\mu$ larger than $2m_\mu$ is $F_\mu$. Furthermore, since all nonzero elements of $S_\mu$ are at least $m_\mu$, the sum of any two such elements must be at least $2m_\mu$. Thus, it remains to show that there are no elements $a,b\in S_\mu$ with $a+b=F_\mu$ and $2m_\mu\in S_\mu$.

Note that, for all $a\in S_\mu$, $a+1\in S_\lambda$. Thus, if $a,b\in S_\mu$, $a+1,b+1\in S_\lambda$. Since $S_\lambda$ is a numerical semigroup, $a+b+2\in S_\lambda$. Since $F_\lambda=F_\mu+2\not\in S_\lambda$, we cannot have $a+b=F_\mu$ for any $a,b\in S_\mu$. 

For the final step, we consider two cases. 
\begin{enumerate}
    \item[Case 1:] Assume that $\lambda_3>2$. In this case, $F_\pi\le |\pi|-1$, so $F_\pi<j-1$. Thus $\max(\N\setminus (S_\mu\cup F_\mu))<2m_\mu$, so $2m_\mu \in S_\mu$. 
    \item[Case 2:] Assume that $\lambda_1>\lambda_2+1$. Note that, if $\lambda_3>2$, we can follow the proof for Case 1, so we only need to consider the case where $\lambda_3=2$. In this case, we note that, if $2m_\mu\not \in S_\mu$, then one of the following statements must be true. 
    \begin{itemize}
        \item There are at least $m_\mu=j-1$ parts of size $2$ in $\lambda$. In this situation, since $|\lambda|=3j$, $\lambda_1<3j-(j-1)-2(j-1)=3$, contradicting the assumption that $\lambda_1>\lambda_2+1$. 
        \item $F_\pi=m_\mu$. In this situation, since $|\lambda|=3j$, $\lambda_1<3j-(j-1)-(j-\lambda_2)-(j-1)=\lambda_2+2$, which again contradicts the assumption that $\lambda_1>\lambda_2+1$. 
    \end{itemize}
\end{enumerate}

\noindent {\bf Proof of 4:} Let $\mu\in \noOnesSet(2j+1)$ such that $F_\mu\ne j+1$. Define $\lambda=\mu\cup\{1^{j-1}\}$. Since $\mu$ has no parts of size $1$, $\lambda$ has exactly $j-1$ parts of size $1$, making $m_{\lambda}=j$. It remains to show that $S_{\lambda}$ is a numerical semigroup. Since $\mu$ has no parts of size $1$ and $\mu_1=\mu_2$, $F_\mu\le \frac{2j+1}{2}+1<j+2$, making $F_{\lambda}<2j+1.$ Furthermore, since $F_\mu\ne j+1$, $F_{\lambda}\ne j+1+(j-1)=2j$. Thus   $F_\lambda<2j=2(m_{\lambda})$, proving that $S_{\tilde\lambda}$ is a numerical semigroup. Therefore, since $|\lambda|=2j+1+(j-1)=3j$, $\lambda$ has exactly $j-1$ parts of size $1$, $\lambda\in \mathcal{M}$ with $\lambda_1=\lambda_2$.  

\noindent {\bf Proof of 5:} Let $\mu\in \noOnesSet(2j+1)$ such that $F_\mu= j+1$. By the definition of $\noOnesSet(2j+1)$, $\mu_3=3$ and $\mu_i=2$ for all $3<i\le k$, with $k$ equal to the length of $\mu$. Define $\lambda:=\{\mu_1+1,\mu_2,\mu_3-1,\mu_4,\ldots,\mu_k,1^{j-1}\}=\{\mu_1+1,\mu_2,2^{k-2},1^{j-1}\}$.  First, note that $|\lambda|=|\mu|+j-1=3j$, $m_\lambda=j$, and $1<\lambda_3<\lambda_2$, so it remains to show that $S_\lambda$ is a numerical semigroup. Note that $F_\lambda=F_\mu+1+j-1=2j+1$. Since $\lambda_1\ne\lambda_2,$ $F_\lambda-1=2j\in S_\lambda$. Thus, the only possible, $a,b\in S_\lambda$ with $a+b\not\in S_\lambda$ are $a=j$ and $b=j+1$. It remains to show that $j+1\not\in S_\lambda$. We will consider two cases based on the number of parts in $\mu$. 
\begin{enumerate}
    \item[Case 1:] Assume that $k=3$. In this case, $\lambda_k=\mu_k-1$. Since $|\mu|=2j+1$, $F_\mu=j+1$, and $\mu_1=\mu_2$, we must have $\mu_1=\mu_2=j-1$. Then, since $|\mu|=\mu_1+\mu_2+\mu_3=2(j-1)+\mu_3$, we know that $\mu_3=3$. Therefore $\lambda_3=2$ and $j+1=m_\lambda+1\notin S_\lambda$. 
    \item[Case 2:] Assume that $k>3$. Since, $F_\mu=j+1$, we have $k-1+\mu_1=j+1$ and $k-2+\mu_1=j$. Since $\mu_k=2<3$,  $j+1=m_\lambda+1\notin S_\lambda$.
\end{enumerate}

\noindent {\bf Proof of 6:} Let $\mu\in\setMult(3j-2,j-1)$. Define $\lambda=\{\mu_1+1,\mu_2,\ldots,\mu_k,1^{j-1}\}$. 

First, note that, because $F_\lambda=F_\mu+2$ and $m_\lambda=m_\mu+1$, $$F_\lambda<2m_\lambda \iff F_\mu<2m_\mu.$$
Thus, we now focus on the case where $F_\mu>2m_\mu$. Define $\pi$ to be the partition obtained by removing the largest hook from $\mu$ (or, equivalently, from $\lambda$). Then 
\begin{align*}
    |\pi|&=|\mu|-F_\mu\\
    &<|\mu|-2m_\mu\\
    &=j.
\end{align*}
Then, since $F_\pi\le |\pi|$, we have that $F_\pi<j$. 

Next, we consider 
\begin{align*}
    \max(\N\setminus(S_\lambda\cup F_\lambda))&=m_\lambda+F_\pi\\
    &<m_\lambda+j\\
    &=2m_\lambda. 
\end{align*}
Therefore, the only possible element of $\N\setminus S_\lambda$ of size at least $2m_\lambda$ is $F_\lambda$. Furthermore, since all nonzero elements of $S_\lambda$ are at least $m_\lambda$, the sum of any two such elements must be at least $2m_\lambda$. We will now show that there are no elements $a,b\in S_\lambda$ with $a+b=F_\lambda$. 

Note that, for all $a\in S_\lambda$, $a-1\in S_\mu$. Thus, if $a,b\in S_\lambda$, $a-1,b-1\in S_\mu$. Since $S_\mu$ is a numerical semigroup, $a+b-2\in S_\mu$. Since $F_\mu=F_\lambda-2\notin S_\mu$, we cannot have $a+b=F_\lambda$ for any $a,b\in S_\lambda$. 

It remains to show that, if $\mu\in \setMult(3j-2,j-1)$, then $\mu_1>\mu_2$ or, $\mu_1=\mu_2$ and $\mu_3\ne 2$. Assume for the sake of contradiction that $\mu_1=\mu_2$ and $\mu_3=2$. We use $k$ to denote the number of parts of $\mu$ that are larger than 1. Then,
\begin{align*}
    |\mu|=2\mu_1+2(k-2)+(j-2).  
\end{align*}
Since $|\mu|=3j-2$, we can see that $$j=\mu_1+k-2.$$
As before, we let $\pi$ be the partition obtained by removing the largest hook from $\mu$. In this case, $F_\pi=\mu_1+k-3=j-1=m_\mu$. Note that $m_\mu+F_\pi\not\in S_\mu$, contradicting the fact that, for $S_\mu$ to be a numerical semigroup, we must have $2m_\mu\in S_\mu.$
 \end{proof}

\section{Conclusion}

Many of the results above involve functions of the form $\countingFuncGenus(L_1(n)-k,L_2(n)-k)$, $\countingFuncMult(L_3(n)-k,L_4(n)-k)$, and $\countingFuncFrob(L_5(n),L_6(n)-k)$ for some specific linear functions $L_i(n)$ and certain values of $k$. As we noted in Remarks \ref{remark:genus}, \ref{remark:Frob} and \ref{remark:mult}, we could extend these results to any specific value of $k$ by explicitly counting exceptional cases. However, it is an open problem to extend the results uniformly to general theorems that hold for any value of $k$. 

 

The solutions to these open problems would involve general forms for the functions  $h(k)$ (described in Remarks \ref{remark:genus} and \ref{remark:mult}) that count exceptional cases. We note that part of the difficulty is that there are often multiple ways of being an `exceptional case', and the functions $h(k)$ may count several different types of exceptions. This is illustrated by the bound given in Remark \ref{remark:Frob} for $\countingFuncFrob(2n,2n-k)$, where we can explicitly identify some exceptional cases, but as $k$ increases, there will be additional exceptions.

\section{Acknowledgements}
This material is based upon work supported by the National Science Foundation under Grant No.~1440140, while the authors were in residence at the Mathematical Sciences Research Institute in Berkeley, California, during the summer of 2022. This work is supported by Korea Institute for Advanced Study (KIAS) grant funded by the Korea government. Hayan Nam was supported by the National Research Foundation of Korea (NRF) grant funded by the Korea government (MSIT) (No. 2021R1F1A1062319).

\printbibliography
\end{document}